\documentclass[11pt, times]{amsart}
\usepackage[utf8x]{inputenc}

\usepackage{graphicx}
\usepackage{amssymb}
\usepackage{amsmath}
\usepackage{amsfonts}
\usepackage{mathrsfs}
\usepackage{enumerate}
\usepackage{multicol}
\usepackage{times}
\usepackage[varg]{txfonts}
\usepackage{hyperref}
\usepackage{yfonts}

\theoremstyle{plain}
\newtheorem{theorem}{Theorem}[section]

\newtheorem{lemma}[theorem]{Lemma}
\newtheorem{corollary}[theorem]{Corollary}

\theoremstyle{definition}

\theoremstyle{remark}
\newtheorem{opp}[theorem]{Open problem}

\begin{document}

\title[Approximate derivations of order $n$]
{Approximate derivations of order $n$}

\author{Eszter Gselmann}
\address{Institute of Mathematics \\ University of Debrecen\\ H-4010 Debrecen\\ P.~O.~Box 12 }
\email{gselmann@science.unideb.hu}


\begin{abstract}
The aim of this paper is to prove characterization theorems for higher order derivations. 
Among others we prove that the system defining higher order derivations is stable. 
Further characterization theorems in the spirit of N.~G.~de Bruijn will also be  presented. 
\end{abstract}

\thanks{This research has been supported by the Hungarian Scientific Research Fund
(OTKA) Grant NK 81402 
and by the T\'{A}MOP 4.2.4.A/2-11-1-2012-0001 (Nemzeti Kiv\'{a}l\'{o}s\'{a}g Program --
Hazai hallgat\'{o}i, illetve kutat\'{o}i szem\'{e}lyi t\'{a}mogat\'{a}st biztos\'{i}t\'{o} rendszer kidolgoz\'{a}sa \'{e}s m\H{u}k\"{o}dtet\'{e}se
konvergencia program) project implemented
through the EU and Hungary co-financed by the European Social Fund.}

\maketitle

\section{Introduction}

Throughout this paper $\mathbb{N}$ denotes the set of the positive integers, further $\mathbb{Z}, \mathbb{Q}$, and
$\mathbb{R}$ have the usual meaning. 

The aim of this work is to prove characterization theorems on derivations as well as on linear functions. 
Therefore, firstly we have to recall some definitions and auxiliary results. 
Unless we indicate in the other way, all the notions follow the monograph Kuczma \cite{Kuc09}. 

A function $f:\mathbb{R}\to\mathbb{R}$ is called an \emph{additive} function
if,
\[
f(x+y)=f(x)+f(y)
\]
holds for all $x, y\in\mathbb{R}$. 

We say that an additive
function $f:\mathbb{R}\to\mathbb{R}$ is a \emph{derivation} if
\[
f(xy)=xf(y)+yf(x)
\]
is fulfilled for all $x, y\in\mathbb{R}$.

The additive function $f\colon\mathbb{R}\to\mathbb{R}$ is termed to be a \emph{linear function} if 
$f$ is of the form 
\[
 f(x)=f(1)x 
\qquad 
\left(x\in\mathbb{R}\right). 
\]

For any function $f\colon \mathbb{R}\to \mathbb{R}$ and $\alpha\in \mathbb{R}$ we define 
\[
 \delta_{\alpha}f(x)=f(\alpha x)-\alpha f(x) 
\qquad 
\left(x\in \mathbb{R}\right). 
\]

Clearly, if $f\colon \mathbb{R}\to \mathbb{R}$ is an additive function then
\[
 \delta_{\alpha}f(x)=0 
\qquad 
\left(\alpha, x\in \mathbb{R}\right)
\]
or 
\[
 \delta_{\alpha}f(1)=0 
\qquad 
\left(\alpha\in \mathbb{R}\right)
\]
yield that $f$ is linear. 

Let $n\in \mathbb{N}$ be fixed. 
Following the work \cite{UngRei98}of J.~Unger and L.~Reich, 
an additive function $f\colon \mathbb{R}\to \mathbb{R}$ is said to be a 
\emph{derivation of order $n$} if, 
\[
f(1)=0 
\quad 
\text{and}
\quad 
 \delta_{\alpha_{1}}\circ \cdots \circ \delta_{\alpha_{n+1}}f(x)=0
\]
is fulfilled for any $x, \alpha_{1}, \ldots, \alpha_{n+1}\in \mathbb{R}$. 

From this notion immediately follows that first order derivations are just real derivations. 
If we drop the assumption $f(1)=0$, we get that an additive function $f\colon \mathbb{R}\to \mathbb{R}$ fulfills 
\[
 \delta_{\alpha}\delta_{\beta}f(x)=0 
\qquad 
\left(\alpha, \beta, x\in \mathbb{R}\right)
\]
if and only if 
\[
 f(x)=d(x)+f(1)x, 
\]
where $d$ is a real derivation.

Furthermore, we also remark that there exists derivations of second order that are not first order derivations. 
Namely, let $d\colon \mathbb{R}\to \mathbb{R}$ be a nontrivial derivation and consider the function $d\circ d$.

In the remainder part of this section, we will summarize those results 
that will be utilized in connection with 
multiadditive and polynomial functions, respectively. 
For further details we refer to Sz\'{e}kelyhidi \cite{Sze85b, Sze85a, Sze91}. 

Let $G, H$ be abelian groups, let $h\in G$ be arbitrary and consider a function $f\colon G\to H$. 
The \emph{difference operator} $\Delta_{h}$ with the span $h$ of the function $f$ is defined by
\[
 \Delta_{h}f(x)=f(x+h)-f(x)
\qquad 
\left(x\in G\right). 
\]
The iterates $\Delta^{n}_{h}$ of $\Delta_{h}$, $n=0, 1, \ldots$ are defined by the recurrence 
\[
 \Delta^{0}_{h}f=f, \qquad 
\Delta^{n+1}_{h}f=\Delta_{h}\left(\Delta^{n}_{h}f\right)
\qquad 
\left(n=0, 1, \ldots\right)
\]
Furthermore, the superposition of several difference operators will be denoted concisely 
\[
 \Delta_{h_{1} \ldots h_{n}}f=\Delta_{h_{1}}\ldots \Delta_{h_{n}}f, 
\]
where $n\in\mathbb{N}$ and $h_{1}, \ldots, h_{n}\in G$. 

Let $n\in\mathbb{N}$ and $G, H$ be abelian groups.
A function $F\colon G^{n}\to H$ is
called \emph{$n$-additive} if, for every
$ i \in \{\, 1 \,,\, 2 \,,\, \dots \,,\, n \,\} $
and for every
$ x_{1}, \ldots, x_{n}, y_{i}\in G \,$,
\begin{multline*}
F\left(x_{1}, \ldots, x_{i-1}, x_{i} + y_{i}, x_{i+1}, \ldots, x_{n}\right)
\\
=F\left(x_{1}, \ldots, x_{i-1}, x_{i}, x_{i+1}, \ldots, x_{n}\right)
+F\left(x_{1}, \ldots, x_{i-1}, y_{i}, x_{i+1}, \ldots, x_{n}\right),
\end{multline*}
i.e., $F$ is additive in each of its variables
$x_{i}\in G$, $i=1, \ldots, n$. 
For the sake of brevity we use the notation $G^{0}=G$ and we call 
constant functions from $G$ to $H$ $0$-additive functions. 
Let $F\colon G^{n}\to H$ be an arbitrary function. 
By the \emph{diagonalization (or trace)} of $F$ we understand the
function
$f\colon G\to H$ arising from
$F$ by putting all the variables (from $G$)
equal:
\[
f(x)=F(x, \ldots, x)
\qquad
\left(x\in G\right).
\]
It can be proved by induction that for any symmetric, $n$-additive function
$F\colon G^{n}\to H$ the equality 
\[
 \Delta_{y_{1}, \ldots, y_{k}}f(x)=
\left\{
\begin{array}{lcl}
 n!F(y_{1}, \ldots, y_{n})& \text{for}& k=n \\
0 & \text{for} & k>n
\end{array}
\right.
\]
holds, whenever $x, y_{1}, \ldots, y_{n}\in G$, where $f\colon G\to H$ denotes the trace of the 
function $F$. This means that a symmetric, $n$-additive function is uniquely determined by its trace. 

The function $f\colon G\to H$ is called a \emph{polynomial function} of degree at most $n$, where 
$n$ is a nonnegative integer, if 
\[
 \Delta_{y_{1}, \ldots, y_{n+1}}f(x)=0
\]
is satisfied for all $x, y_{1}, \ldots, y_{n+1}\in G$. 

\begin{theorem}[Kuczma \cite{Kuc09}, Sz\'{e}kelyhidi \cite{Sze91}]\label{T1.1}
The function $p\colon G\to H$ is a polynomial at degree at most $n$ if and only if there 
exist symmetric, $k$-additive functions $F_{k}\colon G^{k}\to H$, $k=0, 1, \ldots, n$ such that 
\[
 p(x)=\sum_{k=0}^{n}f_{k}(x) 
\qquad 
\left(x\in G\right), 
\]
where $f_{k}$ denotes the trace of the function $F_{k}$, $k=0, 1, \ldots, n$. 
Furthermore, this expression for the function $p$ is unique in the sense that the 
functions $F_{k}$, which are not identically zero,  are uniquely determined. 
\end{theorem}

Henceforth we will say that the
function in question is \emph{locally regular} on its domain, if at least 
one of the following statements are fulfilled. 
\begin{enumerate}[(i)]
 \item bounded on a measurable set of positive measure;
\item continuous at a point;
\item there exists a set of positive Lebesgue measure so that the restriction of the function in question
is measurable in the sense of Lebesgue. 
\end{enumerate}

\begin{lemma}
Let $n\in \mathbb{N}$ be fixed and $p\colon \mathbb{R}\to \mathbb{R}$ be a polynomial of 
degree at most $n$. If $p$ is locally regular then $p$ is a real polynomial, i.e., 
\[
 p(x)=a_{n}x^{n}+\cdots +a_{1}x+a_{0}
\qquad 
\left(x\in \mathbb{R}\right)
\]
with certain real constants $a_{0}, a_{1}, \ldots, a_{n}$. 
\end{lemma}

\begin{corollary}
 Let $n\in \mathbb{N}$ be arbitrarily fixed and 
$F\colon \mathbb{R}\to \mathbb{R}$ be a symmetric, $n$-additive function. 
Assume further, that the trace of $F$ is locally regular, then 
there exists a constant $c\in \mathbb{R}$ such that 
\[
 F(x_{1}, \ldots, x_{n})=cx_{1}\cdots x_{n} 
\qquad 
\left(x\in \mathbb{R}\right). 
\]
\end{corollary}

\section{Results}

\subsection*{Preparatory statements}
We begin with the following characterization of $n$\textsuperscript{th} order derivations.

\begin{theorem}\label{T2.1}
 Let $n\in \mathbb{N}$ be fixed and $f\colon \mathbb{R}\to \mathbb{R}$ be an additive function and assume that the mapping 
\[
 \mathbb{R}\ni \alpha\longmapsto \delta_{\alpha}^{n+1}f(1)
\]
is locally regular. Then and only then there exists an $n$\textsuperscript{th} order derivation $d\colon \mathbb{R}\to \mathbb{R}$ such that 
\[
 f(x)=d(x)+f(1)\cdot x 
\qquad 
\left(x\in \mathbb{R}\right). 
\]
\end{theorem}
\begin{proof}
 Define the function $F\colon \mathbb{R}^{n+1}\to \mathbb{R}$ through 
\[
 F(\alpha_{1}, \ldots, \alpha_{n+1})=\delta_{\alpha_{1}}\circ \cdots \circ \delta_{\alpha_{n+1}}f(1) 
\qquad 
\left(\alpha_{1}, \ldots, \alpha_{n+1}\in \mathbb{R}\right). 
\]
Due to the additivity of $f$, the mapping $F$ is a symmetric, $(n+1)$-additive function. 
Furthermore, its trace, that is, 
\[
 F(\alpha, \ldots, \alpha)=\delta_{\alpha}^{n+1}f(1) 
\qquad 
\left(\alpha\in \mathbb{R}\right)
\]
is locally regular. 
Thus there exists a constant $c\in \mathbb{R}$ so that 
\[
  F(\alpha_{1}, \ldots, \alpha_{n+1})=c\alpha_{1}\cdots \alpha_{n} 
\qquad 
\left(\alpha_{1}, \ldots, \alpha_{n+1}\in \mathbb{R}\right). 
\]
Put $\alpha_{i}=1$ for all $i=1, \ldots, n+1$ to get 
\[
 0=F(1, \ldots, 1)=c. 
\]
This means that the mapping $F$ is identically zero, in other words, we have 
\[
 \delta_{\alpha_{1}}\circ \cdots \circ \delta_{\alpha_{n+1}}f(1)=0
\]
for all $\alpha_{1}, \ldots, \alpha_{n+1}\in \mathbb{R}$. 
This however yields that the function $f-f(1)\mathrm{id}$ is an $n$\textsuperscript{th} order derivation. 
\end{proof}

In view of the previous result, the following corollary can be verified easily. 

\begin{corollary}
 Let $n\in \mathbb{N}$ be fixed and $f\colon \mathbb{R}\to \mathbb{R}$ be an additive function and assume that the mapping 
\[
 \mathbb{R}^{n+1}\ni (\alpha_{1}, \ldots, \alpha_{n+1})\longmapsto \delta_{\alpha_{1}}\circ \delta_{\alpha_{2}}\circ \cdots \circ \delta_{\alpha_{n+1}}f(1)
\]
is locally regular. Then and only then there exists an $n$\textsuperscript{th} order derivation $d\colon \mathbb{R}\to \mathbb{R}$ such that 
\[
 f(x)=d(x)+f(1)\cdot x 
\qquad 
\left(x\in \mathbb{R}\right). 
\]
\end{corollary}
\begin{proof}
 Similarly, as in the proof of the previous theorem, let 
\[
 F(\alpha_{1}, \ldots, \alpha_{n+1})=\delta_{\alpha_{1}}\circ \cdots \circ \delta_{\alpha_{n+1}}f(1) 
\qquad 
\left(\alpha_{1}, \ldots, \alpha_{n+1}\in \mathbb{R}\right). 
\]
Because of the additivity of the function $f$, the mapping $F$ is a locally regular, symmetric and $n$-additive function. 
As in the previous proof, from this we conclude that $f-f(1)\mathrm{id}$ is an $n$\textsuperscript{th} order derivation. 
\end{proof}

\subsection*{Hyers type results}

In this section we will investigate whether the system defining higher order derivations is stable in the sense of 
Hyers. 
To answer this problem affirmatively, we will use the following result of D.~H.~Hyers, see 
\cite{Hye41}.

\begin{theorem}[Hyers]
Let $\varepsilon\geq 0$, $X, Y$ be Banach spaces and
$f:X\to Y$ be a function.
Suppose that
\[
\left\|f(x+y)-f(x)-f(y)\right\|\leq \varepsilon
\]
holds for all $x, y\in X$.
Then,  for all $x\in X$, the limit
\[
a(x)=\lim_{n\to\infty}\frac{f(2^{n}x)}{2^{n}}
\]
does exist, the function $a:X\to\mathbb{R}$ is additive
on $X$, i.e.,
\[
a(x+y)=a(x)+a(y)
\]
holds for all $x, y\in X$, furthermore,
\[
\left\|f(x)-a(x)\right\|\leq \varepsilon
\]
is fulfilled for arbitrary $x\in X$.
Additionally, the function $a:X\to Y$ is
uniquely determined by the above formula.
\end{theorem}

The above theorem briefly expresses the following.
Assume that $X, Y$ are Banach spaces and the function
$f:X\to Y$ satisfies the additive Cauchy equation
only 'approximatively'.
Then there exists a unique additive function
$a:X\to Y$ which is 'close' to the function $f$.
Since 1941 this result has been extended and generalized in a several
ways. Furthermore,
such a problem can obviously be raised concerning not only
the Cauchy equation but also in connection with other equations and system of equations, as well.

In \cite{Bad06} R.~Badora investigated the above stability problem for derivations. 
Furthermore, in \cite{Sch01} J.~Schwaiger proved a stability type result for 
higher order derivations. Concerning stability of real derivations as well as linear functions, we refer to 
\cite{BorGse10, Gse12}. 

Our aim is to generalize the results archived by R.~Badora and J.~Schwaiger and at the same time 
we will show that with a different idea the proofs can be shortened significantly.

Concerning higher order derivations, we will prove the following. 

\begin{theorem}
 Let $n\in \mathbb{N}$ and $\varepsilon\geq 0$. Assume that for the function 
$f\colon \mathbb{R}\to \mathbb{R}$ 
\[
 \left| f(x+y)-f(x)-f(y)\right|<\varepsilon
\qquad 
\left(x\in \mathbb{R}\right)
\]
and that the mapping 
\[
  \mathbb{R}\times \mathbb{R}\ni (\alpha, x)\longmapsto \delta_{\alpha}^{n+1}f(x)
\]
is locally bounded. 
Then there exist an $n$\textsuperscript{th} order derivation $d\colon \mathbb{R}\to \mathbb{R}$ such that 
\[
 \left|f(x)-\left[d(x)+f(1)\cdot x\right]\right| <\varepsilon
\]
for all $x\in \mathbb{R}$. 
\end{theorem}
\begin{proof}
 Since the function $f\colon \mathbb{R}\to \mathbb{R}$ is approximatively additive, that is, 
\[
  \left| f(x+y)-f(x)-f(y)\right|<\varepsilon
\qquad 
\left(x\in \mathbb{R}\right), 
\]
due to the theorem of Hyers we immediately get that there exists an additive function $a\colon \mathbb{R}\to \mathbb{R}$ and a bounded 
function $b\colon \mathbb{R}\to \mathbb{R}$ such that 
\[
 f(x)=a(x)+b(x)
\qquad 
\left(x\in \mathbb{R}\right)
\]
further, 
\[
 \left|b(x)\right|<\varepsilon 
\qquad 
\left(x\in \mathbb{R}\right). 
\]
Additionally, there exists a locally bounded function $B\colon \mathbb{R}^{2}\to \mathbb{R}$ such that 
\[
 \delta_{\alpha}^{n+1}f(x)=B(\alpha, x) 
\qquad 
\left(\alpha, x\in \mathbb{R}\right). 
\]
Therefore, 
\[
 \delta_{\alpha}^{n+1}a(1)= \delta_{\alpha}^{n+1}f(1)-\delta_{\alpha}^{n+1}b(1)=B(\alpha, 1)-\delta_{\alpha}^{n+1}b(1) 
\qquad 
\left(\alpha\in \mathbb{R}\right). 
\]
Let us observe that in the previous expression, the right hand side is a locally bounded function. 
Thus by Theorem \ref{T2.1}. there exists an $n$\textsuperscript{th} order derivation $d\colon \mathbb{R}\to \mathbb{R}$ so that 
\[
 a(x)=d(x)+a(1)x 
\qquad 
\left(x\in \mathbb{R}\right), 
\]
yielding that 
\[
 \left|f(x)-\left[d(x)+\lambda x\right]\right|<\varepsilon 
\qquad 
\left(x\in\mathbb{R}\right), 
\]
where $\lambda=a(1)$. 
\end{proof}

\subsection*{de Bruijn type results}

Finally, we will close this section with a result that concerns also characterization theorem on derivations in the spirit of G.~N.~de Bruijn. 

In \cite{Bru51} he introduced the following notion, for further results we refer to the PhD dissertation of Gy.~Maksa, see \cite{Mak76}.  

The set of real functions $\mathscr{X}\subset \mathbb{R}^{\mathbb{R}}$ is said to have \emph{the difference property} if the following is fulfilled: 
if for all $h\in \mathbb{R}$  the function $\Delta_{h}f\in \mathscr{X}$, then there exists an additive function $a$ and a function 
$\gamma\in \mathscr{X}$ such that 
\[
 f(x)=a(x)+\gamma(x) 
\qquad 
\left(x\in \mathbb{R}\right). 
\]

Subsequently, the following function spaces will be important for us. 

Let $\Omega\subset \mathbb{R}$ be an open set and denote 
\[
 \mathscr{C}^{k}(\Omega)=
\left\{f\colon \Omega\to \mathbb{R}\, \vert \, \text{ $f$ is $k$-times continuously differentiable}\right\}
\]

\[
 \mathscr{D}^{k}(\Omega)=
\left\{f\colon \Omega\to \mathbb{R}\, \vert \, \text{ $f$ is $k$-times differentiable}\right\}
\]

\[
 \mathscr{C}^{\omega}(\Omega)=
\left\{f\colon \Omega\to \mathbb{R}\, \vert \, \text{ $f$ is analytic}\right\}
\]

\[
 \mathscr{R}(\mathbb{R})=
\left\{f\colon \mathbb{R}\to \mathbb{R}\, \vert \, \text{ $f$ is Riemann integrable on every bounded interval}\right\}
\]

\[
 \mathscr{AC}(\mathbb{R})=
\left\{f\colon \mathbb{R}\to \mathbb{R}\, \vert \, \text{ $f$ is absolutely continuous on every bounded interval}\right\}
\]

\[
 \mathscr{BV}(\mathbb{R})=
\left\{f\colon \mathbb{R}\to \mathbb{R}\, \vert \, \text{ $f$ has bounded variation on every bounded interval}\right\}
\]

\begin{theorem}[de Bruijn]
 All of the linear spaces listed above possess the difference property. 
\end{theorem}

For the sake of brevity, if $f\colon \mathbb{R}\to \mathbb{R}$ is a function and $n\in \mathbb{N}$, 
then we will write 
\[
 (\mathbf{D}f)(\alpha)=\delta_{\alpha}^{n+1}f(1)  
\qquad 
\left(\alpha\in \mathbb{R}\right). 
\]

In contrast to the statement of the previous subsection, here we will assume not that
the Cauchy difference is bounded, but we suppose that all the first order differences of the function in question 
are regular.

\begin{theorem}
 Let $n\in \mathbb{N}$ be arbitrarily fixed and 
$f\colon \mathbb{R}\to \mathbb{R}$ be a function. 
Assume furthermore at least one of the following statements are fulfilled. 
\begin{enumerate}[(i)]
 \item there exists $k\in \mathbb{N}\cup \left\{0\right\}$ such that for all $h\in \mathbb{R}$ 
$\Delta_{h}f  \in \mathscr{C}^{k}(\mathbb{R})$ and $\mathbf{D}f$ continuous at least one point. 
 \item there exists $k\in \mathbb{N}$ such that for all $h\in \mathbb{R}$ 
$\Delta_{h}f \in \mathscr{D}^{k}(\mathbb{R})$  and $\mathbf{D}f$ continuous at least one point.
 \item  for all $h\in \mathbb{R}$ 
$\Delta_{h}f \in \mathscr{C}^{\omega}(\mathbb{R})$  and $\mathbf{D}f$ continuous at least one point.
 \item  for all $h\in \mathbb{R}$ $\Delta_{h}f \in \mathscr{R}(\mathbb{R})$  and $\mathbf{D}f$ fulfills regularity condition (i) or (iii). 
 \item  for all $h\in \mathbb{R}$ $\Delta_{h}f(x) \in \mathscr{AC}(\mathbb{R})$ and $\mathbf{D}f$ fulfills regularity condition (i) or (iii).
 \item  for all $h\in \mathbb{R}$ $\Delta_{h}f(x) \in \mathscr{BV}(\mathbb{R})$ and $\mathbf{D}f$ fulfills regularity condition (i) or (iii).
\end{enumerate}
Then and only then there exists an $n$\textsuperscript{th} order derivation $d\colon \mathbb{R}\to \mathbb{R}$ and a function 
$\gamma\colon \mathbb{R}\to \mathbb{R}$  such that 
\[
 f(x)=d(x)+\lambda x +\gamma(x)
\]
for all $x\in \mathbb{R}$, here the function $\gamma$ is, according to the above cases $k$-times continuously differentiable, 
$k$-times differentiable, analytic, Riemann integrable, absolutely continuous and of bounded variation, respectively. 
\end{theorem}
\begin{proof}
 Assume that at least one of (i), (ii) and (iii) is fulfilled and let 
$\textgoth{D}$  denote any of the spaces appearing in this statements. 
Since $\textgoth{D}$ has the difference property we get that 
there exists an additive function $a\colon \mathbb{R}\to \mathbb{R}$ and $\gamma\in \textgoth{D}$ such that 
\[
 f(x)=a(x)+\gamma(x) 
\qquad 
\left(x\in \mathbb{R}\right). 
\]
On the other hand,  the mapping $\mathbf{D}f$  is continuous at least one point. 
Thus there exists a function $R\colon \mathbb{R}\to \mathbb{R}$ which is continuous at least one point so that 
\[
 (\mathbf{D}f)(\alpha)=R(\alpha) 
\qquad 
\left(\alpha\in \mathbb{R}\right). 
\]
This implies that for the additive function $a$, 
\[
 \delta_{\alpha}^{n+1}a(1)=R(\alpha)-\delta^{n+1}_{\alpha}\gamma(1) 
\qquad 
\left(\alpha\in \mathbb{R}\right). 
\]
Since $\gamma\in \textgoth{D}$, the function $\gamma$ is continuous everywhere. Therefore the right hand side in the previous representation in 
continuous at least one point. 
Thus by Theorem \ref{T2.1}, there exists an $n$\textsuperscript{th} order derivation $d\colon \mathbb{R}\to \mathbb{R}$ such that 
\[
a(x)=d(x)+a(1)x 
\qquad 
\left(x\in \mathbb{R}\right), 
\]
from which we get that 
\[
 f(x)=d(x)+\lambda x+\gamma(x) 
\qquad 
\left(x\in \mathbb{R}\right), 
\]
where $\lambda =a(1)$.

Now assume that at least one of (iv), (v) and (vi) is fulfilled and let 
\[\textgoth{L}\in \left\{\mathscr{R}(\mathbb{R}, \mathscr{AC}(\mathbb{R}), \mathscr{BV}(\mathbb{R}))\right\}.\]
Since the linear space $\textgoth{L}$ has the finite difference property, we immediately get that there exists an 
additive function $a\colon \mathbb{R}\to \mathbb{R}$ and $\gamma\in \textgoth{L}$ such that 
\[
 f(x)=a(x)+\gamma(x) 
\qquad 
\left(x\in \mathbb{R}\right). 
\]
On the other hand there exists a function $R\colon \mathbb{R}\to \mathbb{R}$ fulfilling regularity assumption (i) or (ii) such that 
\[
 \delta_{\alpha}^{n+1}f(1)=R(\alpha) 
\qquad 
\left(\alpha \in \mathbb{R}\right). 
\]
For the additive function $a$ this yields that 
\[
  \delta_{\alpha}^{n+1}a(1)=R(\alpha)-\delta^{n+1}_{\alpha}\gamma(1) 
\qquad 
\left(\alpha\in \mathbb{R}\right). 
\]
Since any element of the linear space $\textgoth{L}$ fulfills (i) as well as (ii) the right hand side of the above representation satisfies 
regularity assumption (i) or (iii). 
Applying Theorem \ref{T2.1}, there exists an $n$\textsuperscript{th} order derivation 
$d\colon \mathbb{R}\to \mathbb{R}$ such that 
\[
 a(x)=d(x)+a(1)x 
\qquad 
\left(\in \mathbb{R}\right). 
\]
From which the statement follows. 
\end{proof}

\section{A remark}

Due to a result of M.~Laczkovich (see \cite{Lac80}), the set of all measurable real functions does not have the difference property. 
More precisely, the following statement holds. 

\begin{theorem}
 Let $f\colon \mathbb{R}\to \mathbb{R}$ be function such that for all 
$h\in \mathbb{R}$ the function $\Delta_{h}f$ is measurable. 
Then there exist and additive function $a\colon \mathbb{R}\to \mathbb{R}$, 
a measurable function $\gamma\colon \mathbb{R}\to \mathbb{R}$ and a function 
$\sigma\colon \mathbb{R}\to \mathbb{R}$ so that for almost all $x, h\in \mathbb{R}$, 
\[
 \sigma(x+h)=\sigma(x)
\]
and 
\[
 f(x)=a(x)+\sigma(x)+\gamma(x) 
\qquad 
\left(x\in \mathbb{R}\right). 
\]
\end{theorem}

Assuming the Continuum Hypothesis, there exists a \emph{nonmeasurable} function $\sigma\colon \mathbb{R}\to \mathbb{R}$ with the above properties. 
This shows that the set of measurable functions does not have the difference property. 

Nevertheless, if we assume additionally that the mapping $\mathbf{D}f$ is also measurable then it can happen that the 
function $\sigma$ in the theorem of Laczkovich should be measurable. 
Thus we can formulate the following problem. 

\begin{opp}
Let $n\in \mathbb{N}$ be fixed and $f\colon \mathbb{R}\to \mathbb{R}$ be a function. 
Assume that for all $h\in \mathbb{R}$ the function $\Delta_{h}$ is measurable, additionally the 
mapping 
\[
 \mathbb{R}\ni \alpha \longmapsto \delta^{n+1}_{\alpha}f(1)
\]
is measurable, as well. 
Is it true that there exists a derivation $d\colon \mathbb{R} \to \mathbb{R}$ of order $n$ and a measurable 
function $\gamma\colon \mathbb{R}\to \mathbb{R}$ so that 
\[
 f(x)=d(x)+\gamma(x) 
\]
for all $x\in \mathbb{R}$. 
\end{opp}


\end{document}